\documentclass{article}

\usepackage{latexsym,amsfonts,amssymb,epsfig,verbatim,}
\usepackage{amsmath,amsthm,amssymb,latexsym,graphics,textcomp}
\usepackage{enumerate}
\usepackage{graphicx}
\usepackage{url}
\newtheorem{theorem}{Theorem}[section]
\newtheorem{lemma}[theorem]{Lemma}

\newtheorem{corollary}[theorem]{Corollary}

\newtheorem{conjecture}[theorem]{Conjecture}

\title{The minimum color degree and a large rainbow cycle in an edge-colored graph}
\author{Wipawee Tangjai}
\date{}
\begin{document}

\maketitle

\begin{abstract}
Let $G$ be an edge-colored graph with $n$ vertices.
A subgraph $H$ of $G$ is called a {\it rainbow subgraph} of $G$ if  the colors of each pair of the edges in $E(H)$ are distinct.
We define the {\it minimum color degree} of $G$ to be the smallest number of the
 colors of the edges that are incident to a vertex $v$, for all $v\in V(G)$.
Suppose that $G$ contains no rainbow-cycle subgraph of length four.
We show that if the minimum color degree of $G$ is at least $\frac{n+3k-2}{2}$, then $G$ contains a rainbow-cycle subgraph of length at least $k$, where $k\geq 5$.
Moreover, if the condition of $G$ is restricted to a triangle-free graph that contains a rainbow path of length at least $\frac{3k}{2}$, then the lower bound of the minimum color degree of $G$ that guarantees an existence of a rainbow-cycle subgraph of length to at least $k$ can be reduced to $\frac{2n+3k-1}{4}$.
\end{abstract}

\section{Introduction}

For a finite simple undirected graph $G$, where $G=(V(G),E(G))$ and $|V(G)|=n$, we define an edge coloring $c$ of $G$ to be a function, where $c:E(G)\rightarrow {\mathbb Z}^+$. 
For a subgraph $H$ of $G$, the edge coloring of $H$ is the restriction of $c$ to $E(H)$.
If the colors of each pair of the edges in $E(H)$ are distinct, then $H$ is called a {\it rainbow subgraph} or a {\it heterochromatic subgraph} of $G$.
The works related to rainbow subgraphs in various types including paths, trees and cycles appear in the survey by M. Kano and X. Li \cite{Kano2008}.

In this paper, we are interested in a condition of an existence of a large rainbow-cycle subgraph in $G$.
In 2005, J.J. Montellano-Ballesteros and V. Neumann-Lara \cite{Montellano-Ballesteros2005} solved the conjecture of Erd\H{o}s, Simonovits and S{\'o}s on the rainbow cycle. 
They gave a condition of the number of the colors in a complete graph $K_n$ that guarantees an existence of a rainbow-cycle subgraph.
Let $c(E(G))$ be the set of the colors of the edges appearing in $G$.
They showed that if $|c(E(K_n))|=n\left(\frac{k-2}{2}+\frac{1}{k-1}\right)+O(1)$, then $G$ contains a rainbow cycle of length at least $k$.
Meanwhile, H. J. Broesma et. al. \cite{BroersmaLiWoeZh2005} 
showed that if $|c(E(G))|\geq n$, then $G$ contains a rainbow cycle of length at least $\frac{2|c(E(G))|}{n-1}$.

In 2012, H. Li and G. Wang \cite{Li20121958} took a different approach and studied the existence of the rainbow-cycle subgraph of $G$ by considering its {\it minimum color degree} $\delta^c(G)$ of $G$, which is the smallest number of all distinct colors of the edges that are incident to a vertex $v$, for all $v\in V(G)$.
In Theorem \ref{thm: LiWang}, H. Li and G. Wang showed that, for a triangle-free graph $G$ with at least eight vertices, if $\delta^c(G)\geq \frac{3}{4}n+1$, then $G$ contains a rainbow cycle of length at least $\delta^c(G)-\frac{3}{4}n+2$.
The largest lower bound of the guaranteed length of the rainbow-cycle subgraph of $G$ in Theorem \ref{thm: LiWang} is at most $\frac{n}{4}+1$.

\begin{theorem}\cite{Li20121958}\label{thm: LiWang} Let $G$ be a triangle-free graph with $n$ vertices, where $n\geq 8$. If $\delta^c(G)\geq \frac{3}{4}n+1$, then $G$ contains a rainbow cycle of length at least $\delta^c(G)-\frac{3}{4}n+2$.
\end{theorem}

In 2016, R. \v{C}ada, A. Kaneko and Z. Ryj\'{a}\v{c}ek \cite{Cada20161387} gave a sufficient condition of the minimum color degree $\delta^c(G)$ of $G$ so that $G$ contains a rainbow-cycle subgraph of length at least four in Theorem \ref{thm: Cada2}.
We adopt the notations used by \v{C}ada et. al. \cite{Cada20161387}.
For a pair of vertices $u,v\in V(G)$, where $uv\in E(G)$, let $c(uv)$ be the color of the edge $uv$.
For a subgraph $H$ of $G$, the notation $c(u,H)$ is the set of the colors of the edges joining the vertex $u$ and a vertex in $H$.
Let $P$ be a path $u_1u_2\dots u_p$ and let $u_i P u_j$ be a subpath of $P$ that starts at $u_i$ and ends at $u_j$, where $ i,j\in \{1,\dots,p\}$.
Throughout this paper, we let $G$ be an edge-colored graph with an edge-coloring $c$ and $|V(G)|=n$.

\begin{theorem}\cite{Cada20161387}\label{thm: Cada2}
If $\delta^c(G)>\frac{n}{2}+2$, then $G$ contains a rainbow cycle of length at least four.
\end{theorem}

Lemma \ref{lem: Cada1} and \ref{lem: Cada2} are used to prove Theorem \ref{thm: Cada2}.
We will later use these lemmas to prove the main theorem.

\begin{lemma}\label{lem: Cada1}\cite{Cada20161387}
For a graph $G$, let $P=u_1\dots u_p$ be the longest rainbow path of $G$.
If $G$ contains no rainbow cycle of length at least $k$, where $k\leq p$, then
for any color $a\in c(u_1,u_kPu_p)$ and vertex $u_i\in V(u_kPu_p)$, where $c(u_1u_i)=a$, there is an edge $e\in E(u_1Pu_i)$ such that $c(e)=a$.
\end{lemma}

\begin{lemma}\label{lem: Cada2} \cite{Cada20161387}
For a graph $G$, let $P=u_1\dots u_p$ be the longest rainbow path of $G$. If $G$ contains no rainbow cycle of length at least $k$, where $k\leq p$, then for any positive integers $s$, $t$ such that $s+t=k$,
\[
|c(u_1,u_kPu_{p-(t-1)})\cap c(u_p,u_sPu_{p-(k-1)})|\leq 1.
\]
\end{lemma}

From Theorem \ref{thm: Cada2}, \v{C}ada et. gave the following conjecture.

\begin{conjecture}\label{conj: Cada}
If $\delta^c(G)\geq \frac{n+k}{2}$, then $G$ contains a rainbow-cycle subgraph of lenght at least $k$. 
\end{conjecture}

In this work, by using the method appearing in Theorem \ref{thm: Cada2}, we give a progress toward Conjecture \ref{conj: Cada}.
In our main theorem, we showed that if $G$ does not contain a rainbow cycle of length four and $\delta^C(G)\geq \frac{n+3k-2}{2}$, then $G$ contains a rainbow cycle of length at least $k$.
In Section \ref{sec: discuss}, we will discuss the result in the main theorem in comparison with Theorem \ref{thm: LiWang}.
We showed that by restricting the condition of Theorem \ref{thm: LiWang} to be rainbow-$C_4$-free, we can ignore the condition that the graph is triangle-free and the length of the guaranteed rainbow-cycle subgraph can be at least $\frac{k}{2}-\frac{n+4}{4}$ larger than in Theorem \ref{thm: LiWang}.

In order to apply Theorem \ref{thm: main}, we require an existence of a long rainbow path.
In 2014, A. Das, S. V. Subrahmanya and P. Suresh \cite{Das2014} gave a lower bound of the length of the longest rainbow path in the term of the lower bound of the minimum color degree in Theorem \ref{thm: Das2014}.

\begin{theorem}\cite{Das2014}\label{thm: Das2014}
Let $G$ be an edge-colored graph, where $\delta^c(G)\geq t$ and $t\geq 8$.
The maximum length of the rainbow paths in $G$ is at least $\left\lceil \frac{3t}{5}\right\rceil$.
\end{theorem}

In 2016, H. Chen and X. Li gave a larger lower bound of the length of the longest rainbow path with respect to the lower bound of $\delta^c(G)$.

\begin{theorem}\cite{Chen2016}\label{thm: Chen2016}
Let $G$ be an edge-colored graph. If $\delta^c(G)\geq t\geq 7$, then $G$ contains a rainbow path of length at least $\left\lceil\frac{2t}{3}\right\rceil+1$.
\end{theorem}

\section{Main Results}

Theorem \ref{thm: main} is obtained by using the method in Theorem \ref{thm: Cada2} with some generalization.
As a result, Theorem \ref{thm: main} gives a lower bound of $\delta^c(G)$ guaranteeing a rainbow-cycle subgraph of length at least $k$, where $k\geq 5$, in a graph containing no rainbow-cycle subgraph of length four.
The largest length of the guaranteed rainbow-cycle subgraph of $G$ in Theorem \ref{thm: LiWang} is at most $\frac{n}{4}+1$, while, it is $\frac{n+4}{3}$ in
Theorem \ref{thm: main}.

\begin{theorem}\label{thm: main}
Let $G$ be a graph with no rainbow-cycle subgraph of length four.
If $\delta^C(G)\geq \frac{n+3k-2}{2}$, then $G$ contains a rainbow-cycle subgraph of length at least $k$, where $k\geq 5$.
\end{theorem}

\begin{proof}
For a fixed $k\geq 5$.
Suppose that $G$ contains no rainbow-cycle subgraph of length at least $k$.
Since $|V(G)|\geq \delta^c+1$, it follows that $n\geq 3k$.
Let $P=u_1u_2\dots u_p$ be the longest rainbow path in $G$.
By Theorem \ref{thm: Chen2016}, it follows that 
\[
p\geq \left\lceil\frac{(n+3k-4)}{3}\right\rceil+1 \geq 2k.
\]
Let $s,t\in {\mathbb N}$ be such that $s=\left\lfloor \frac{k}{2}\right\rfloor $ and $t=\left\lceil \frac{k}{2}\right\rceil$.
So $p-(t-1)>k$ and $p-(k-1)>s$.
Let $A=c(u_1,u_kPu_{p-(t-1)})$ and $B=c(u_p,u_sPu_{p-(k-1)})$.
By Lemma \ref{lem: Cada2}, we have $|A\cap B|\leq 1$.
Let $P^C$ be the subgraph of $G$ induced by $V(G)-V(P)$ and let
\[
C_0=(c(u_1,P^C)\setminus c(u_1,P))\cap (c(u_p,P^C)\setminus c(u_p,P)).
\]
So $C_0\cap (A\cup B)=\emptyset$.
By Lemma \ref{lem: Cada1}, if $a\in A\cup B$, then $a\in c(E(P))$.
If $a\in C_0$, then $a\in c(E(P))$; otherwise, there exists a longer rainbow path.

By Lemma \ref{lem: Cada1}, if $c(u_1u_2)\in B$, then there exists an edge $e\in E(u_s P u_{p})$ such that $c(e)=c(u_1u_2)$; hence, the path $P$ is not rainbow which is a contradiction.
So, $c(u_1u_2)\not\in B$ and, similarly, $c(u_{p-1}u_p)\not\in A$.
Since $P$ is rainbow and $G$ contains no rainbow-cycle subgraph of length at least $k$, it follows that if $u_1u_p\in E(G)$, then $c(u_1u_p)\in c(E(P))$.
Let 
\begin{align*}
\epsilon_1&=
\begin{cases}
1, &\text{ if } c(u_1u_2)\not\in A,\\
0, &\text{ otherwise},
\end{cases}
\\
\epsilon_2&=
\begin{cases}
1, &\text{ if } c(u_{p-1}u_p)\not\in B,\\
0, &\text{ otherwise,}
\end{cases}
\\
\epsilon_1'&=
\begin{cases}
1, &\text{ if } c(u_1u_p)\not\in A\cup c(u_1u_2),\\
0, &\text{ otherwise,}
\end{cases}
\\
\epsilon_2'&=
\begin{cases}
1, &\text{ if } c(u_1u_p)\not\in B\cup c(u_pu_{p-1}),\\
0, &\text{ otherwise.}
\end{cases}
\end{align*}
So,
\begin{align}
|V(P)|&=|E(P)|+1\notag\\
&\geq |c(E(P))|+1\notag\\
&\geq |A\cup B|+|C_0|+\epsilon_1+\epsilon_2+\epsilon_1'\epsilon_2'+1\notag\\
&\geq |A|+|B|+|C_0|+\epsilon_1+\epsilon_2+\epsilon_1'\epsilon_2'.\label{eq: VP}
\end{align}

Let $C_1=c(u_1,P^C)\setminus (C_0\cup c(u_1,P))$ and $C_2=c(u_p,P^C)\setminus (C_0\cup c(u_p,P))$.
For each $a\in C_1\cup C_2$, we have $a\in c(E(P))$; otherwise, there exists a longer rainbow path.
By the construction of $C_0,C_1$ and $C_2$, it follows that
\[
C_1\cap C_2 =((c(u_1,P^C)\setminus c(u_1,P))\cap (c(u_1,P^C)\setminus c(u_p,P)))\setminus C_0=\emptyset.
\]

Suppose $C_1=\{c_1,\dots, c_{|C_1|}\}$ and $C_2=\{c_1',\dots, c_{|C_1|}'\}$.
Let $X_{c_i}$ be a subset of $N_{P^C}(u_1)$, where $X_{c_i}=\{v\in N_{P^C}(u_1): c(u_1v)=c_i\}$, for all $i\in \{1,\dots, |C_1|\}$.
Similarly, let $X_{c_j'}$ be a subset of $N_{P^C}(u_p)$, where $X_{c_j'}=\{v\in N_{P^C}(u_p): c(u_pv)=c_j'\}$, for all $j\in \{1,\dots, |C_2|\}$.
Next, we choose one vertex $x_i$ from each $X_{c_i}$ and one vertex $y_j$ from $X_{c_j'}$, where $i\in \{1,\dots, |C_1|\}$ and $j\in \{1,\dots, |C_2|\}$.
We have $\{x_1,\dots,x_{|C_1|}\}$ and $\{y_1,\dots,y_{|C_2|}\}$.
Since  $C_1\cap C_2=\emptyset$ and $G$ contains no rainbow-cycle subgraph of length four, it follows that 
\[
|\{x_1,\dots,x_{|C_1|}\}\cap \{y_1,\dots,y_{|C_2|}\}|\leq 1.
\]
So,
\begin{equation}\label{eq: VPC}
|V(P^C)|\geq |C_0|+|C_1|+|C_2|-1.
\end{equation}

We note that $|c(u_1,P)|\leq p-1$ and $|A\cup \{c(u_1u_2)\}|\leq p-t-k+3$.
Suppose $|A\cup \{c(u_1u_2)\}|=l$ and $l'=(p-t-k+3)-l$.
It follows that the number of the colors in $c(u_1,P)$ is also less than $|E(P)|$ by at least $l'$.
So, 
\[
|c(u_1,P)|\leq (p-1)-l'=t+k-4+l.
\]
Hence,
\begin{align}\label{eq: bound_cu1}
|c(u_1,P)\setminus A\cup \{c(u_1u_2)\}|
&=|c(u_1,P)|-|A\cup \{c(u_1u_2)\}|\notag\\
&\leq (t+k-4+l)-l\notag\\
&=t+k-4.
\end{align}
Analogously, since $|B\cup \{c(u_{p-1}u_p)\}|\leq p-s-k+2$, it follows that 
\begin{equation}\label{eq: bound_cu2}
|c(u_p,P)\setminus (B\cup \{c(u_{p-1}u_p\})|\leq s+k-3.
\end{equation}
Next, we consider the number of the colors of the edges that are incident to $u_1$.
Since, $A,\,C_0,\,C_1$ and $c(u_1,P)\setminus A\cup \{c(u_1u_2)\}$ are all disjoint, it follows that
\begin{equation}
|A|+|C_0|+|C_1|+\epsilon_1+\epsilon_1'+(k+t-4)\geq d^c(u_1)\geq \delta^c(G).
\end{equation}
So, if we omit $\epsilon_1'$, then
\begin{equation}\label{eq: u1}
|A|+|C_0|+|C_1|+\epsilon_1\geq \delta^c(G)-k-t+3.
\end{equation}
Similarly,
\[
|B|+|C_0|+|C_2|+\epsilon_2+\epsilon_2'+k+s-3 \geq d^c(u_p)\geq  \delta^c(G).
\]
Hence, if we omit $\epsilon_2'$, then
\begin{equation}\label{eq: u2}
|B|+|C_0|+|C_2|+\epsilon_2\geq \delta^c(G)-k-s+2.
\end{equation}
By (\ref{eq: VP}), (\ref{eq: VPC}), (\ref{eq: u1}) and (\ref{eq: u2}),
\begin{align*}
|V(P)|+|V(P^C)|
&\geq (|A|+|C_0|+|C_1|+\epsilon_1)+(|B|+|C_0|+|C_2|+\epsilon_2)-1+\epsilon_1'\epsilon_2'\\
&\geq (\delta^c-k-t+3)+(\delta^c-k-s+2)-1\\
&=2\delta^c-2k-t-s+4\\
&=2\delta^c-3k+4 >n, 
\end{align*}
which is a contradiction.
Therefore, $G$ contains a rainbow-cycle subgraph of length at least $k$.
\end{proof}

If $G$ is triangle-free, then  the bound of sizes of $c(u_1,P)$, $A\cup \{c(u_1u_2)\}$ and 
$B\cup \{c(u_{p-1}u_p)\}$ can be reduced to
\begin{itemize}
\item $|c(u_1,P)|\leq \left\lceil\frac{p-1}{2}\right\rceil$
\item $|A\cup \{c(u_1u_2)\}|\leq \left\lceil \frac{p-t-k+3}{2}\right\rceil$
\item $|B\cup \{c(u_{p-1}u_p)\}|\leq \left\lceil\frac{p-s-k+2}{2}\right\rceil$.
\end{itemize}
Hence, (\ref{eq: bound_cu1}) and (\ref{eq: bound_cu2}) can be reduced to 

\begin{equation}
|c(u_1,P)\setminus A\cup \{c(u_1u_2)\}|\leq \frac{t+k-2}{2}
\end{equation}

and 
\begin{equation}
|c(u_p,P)\setminus (B\cup \{c(u_{p-1}u_p\})|\leq \frac{s+k-1}{2}.
\end{equation}
Thus, if $G$ is triangle-free with no rainbow-cycle subgraph of length four, then the lower bound of $\delta^c(G)$ can be reduced to $\frac{2n+3k-1}{4}$; however, we need a condition of the existence of a rainbow path of length at least $\frac{3k}{2}$ in $G$.

\begin{theorem}\label{thm: main2}
Let $G$ be a triangle-free graph with no rainbow-cycle subgraph of length four.
If $G$ contains a rainbow path of length at least $\frac{3k}{2}$ and $\delta^C(G)\geq \frac{2n+3k-1}{4}$, then $G$ contains a rainbow cycle of length at least $k$, where $k\geq 5$.
\end{theorem}

We note that if we omit the condition of the length of the longest rainbow path in Theorem \ref{thm: main2}, then the condition of only the minimum degree is not able to guarantee the existence of the needed rainbow path.
To omit such condition, we combine Theorem \ref{thm: Chen2016} and Theorem \ref{thm: main} which result to Corollary \ref{cor: main} as follows.

\begin{corollary}\label{cor: main}
If $n\geq 3k+1$ and $\delta^c(G)\geq \frac{2n+3k-1}{4}$, then $G$ contains a rainbow cycle of length at least $k$.

\end{corollary}
\begin{proof}
By Theorem \ref{thm: Chen2016}, there exists a path $u_1\dots u_p$, where 
\[
p\geq \left\lceil\frac{2n+3k-1}{6}\right\rceil+1 > \frac{3k}{2}.
\]
Hence, by Theorem \ref{thm: main}, there exists a rainbow-cycle subgraph of length at least $k$.
\end{proof}
\section{Discussion}\label{sec: discuss}
In this section, we compare the length of the rainbow-cycle subgraphs guaranteed by Theorem \ref{thm: LiWang} and Theorem \ref{thm: main}.
We consider a graph $G$ satisfying the conditions in both theorems.
The maximum length of the guaranteed rainbow-cycle subgraph in Theorem \ref{thm: LiWang} is at most $\frac{n}{4}+1$, whereas, the guaranteed length of the rainbow-cycle subgraph of $G$ in Theorem \ref{thm: main} can be up to $\frac{n}{3}$.
Let $k\in {\mathbb N}$ be such that $\delta^c(G)=\left\lceil\frac{n+3k-2}{2}\right\rceil$.
Theorem \ref{thm: main} implies that $G$ contains a rainbow-cycle subgraph of length at least $k$.
Next, we show that the guaranteed length of the rainbow-cycle subgraph obtained by Theorem \ref{thm: LiWang} of such graph is less than $k$.
Since $\delta^c(G)=\left\lceil\frac{n+3k-2}{2}\right\rceil$, it follows that 
 \[
\delta^c(G)=\frac{n+3k-1}{2} \text{ or } \delta^c(G)= \frac{n+3k-2}{2}.
\]
So,
\[
 \delta^c(G)=\left(\frac{3n}{4}+1\right)+\frac{6k-n-6}{4} \text{ or } \delta^c(G)= \left(\frac{3n}{4}+1\right)+\frac{6k-n-8}{4}.
\]

We note that if $k<\frac{n+6}{6}$, then Theorem \ref{thm: LiWang} is not applicable.
We consider case $k\geq \frac{n+6}{6}$.
By Theorem \ref{thm: LiWang}, $G$ contains a rainbow cycle of length at least $\delta^c(G)-\frac{3}{4}n+2$, which is either 
\[
k+\left(\frac{k}{2}-\frac{n+6}{4}\right) \text{ or } k+\left(\frac{k}{2}-\frac{n+4}{4}\right),
\]
with respect to $\delta^c(G)$. 
In order to guarantee a larger length of a rainbow cycle in $G$, the value of $k$ in Theorem \ref{thm: LiWang} has to be larger than $\frac{n}{2}$, and hence, $\delta^c(G)-\frac{3}{4}n+2>\frac{n}{2}$  which is not possible, because the maximum guaranteed length of the rainbow cycle from Theorem \ref{thm: LiWang} is at most $\frac{n}{4}-1$.
Therefore, Theorem \ref{thm: main} guarantees an existence of a larger length of a rainbow-cycle subgraph in $G$ by at least $\frac{n+4}{4}-\frac{k}{2}$.
However, in order to apply  Theorem \ref{thm: main}, the graph $G$ cannot contain a rainbow cycle of length four, whereas, this condition is not necessary in Theorem \ref{thm: LiWang}.

The result in Theorem \ref{thm: main} is a progress toward the Conjecture \ref{conj: Cada} given by R. \v{C}ada, A.Kaneko and Z. Ryj\'{a}\v{c}ek \cite{Cada20161387}.
However, the lower bound of $\delta^c(G)$ is still larger than the conjectured bound. 
We also note that Theorem \ref{thm: main} is not a generalization of Theorem \ref{thm: Cada2} because of the exclusion of the rainbow-cycle subgraph of length four.

\section{Acknowledgement}
I would like to thank Kittikorn Nakprasit for a suggestion on this paper and Chokchai Viriyapong for helping with LaTeX technicality.
This research was financially supported by National Science and Technology Development Agency of Thailand with grant number: SCH-NR2016-531.

\bibliographystyle{plain}
\bibliography{science}

\end{document}